\documentclass[draft]{article}
\def\today{21.9.2012} 
\usepackage{amsmath,amsfonts,amsthm,amssymb,amscd}

\theoremstyle{plain} \newtheorem{theorem}{Theorem}[section]
\newtheorem{lemma}[theorem]{Lemma}
\newtheorem{proposition}[theorem]{Proposition}
\newtheorem{corollary}[theorem]{Corollary} \theoremstyle{definition}
 \theoremstyle{remark}
\newtheorem{remark}[theorem]{Remark}

\newcommand{\R}{{\mathbb R}}

\def\im{{\rm i}}    

\numberwithin{equation}{section}

\begin{document}

\title{Decay and scattering of small solutions of pure power
NLS in  ${\mathbf R}$ with $p>3$  and with a potential  }

\author {Scipio Cuccagna, Vladimir Georgiev and Nicola Visciglia}

\date{\today}
\maketitle

\begin{abstract} We prove decay and scattering of solutions of the Nonlinear Schr\"oding-er equation
(NLS) in ${\mathbf R}$
with pure power nonlinearity with exponent $3<p<5$ when the initial datum is small
in $\Sigma$ (bounded energy and variance), in the presence of a linear
inhomogeneity  represented by a linear potential which is a real valued Schwarz
function. We assume absence of  discrete modes. The proof is analogous to the one for the translation invariant equation.
In particular we find    appropriate operators commuting with the linearization.

\end{abstract}

\section{Introduction}

We consider
\begin{equation}\label{eq:PotNLS}
 ({\bf i} \partial _t  + \triangle _V ) u+ \lambda
 |u| ^{p-1} u=0  \text{ for $t\ge 1$,  $x\in {\mathbf R}  $  and $u(1) = u_0$}
\end{equation}
with $\triangle _V :=\triangle     -V(x)$ and $ \triangle :=\partial^2_x $ and $\lambda \in {\mathbf R}\backslash \{ 0\}$. In this paper we focus on   exponents
$3<p<5$.   $V$ is a real valued Schwartz function and $\triangle _V $
is taken without eigenvalues.

It is well known that for $2\le p < 5$ the initial value problem  in
\eqref{eq:PotNLS} is globally well posed in $H^1({\mathbf R} )$.
Our goal is to study the   asymptotic behavior of  solutions with  initial data
$
    u(1) = u_0 $  of size $\epsilon$  in a   suitable Sobolev norm, with
    $\epsilon$
    sufficiently small.
 It is natural to  ask whether such solutions are   asymptotically free and
 satisfy
 \begin{equation}\label{eq:dispersion}
  \| u(t) \| _{L^\infty ({\mathbf R})} \le   {C_0}{t^{-\frac{1}{2}  }}  \epsilon  ,
\end{equation}
 that is have the    decay rate of the solution to the linear Schr\"odinger equation.

 We recall  some of the results for $V=0$.  For spatial dimension
 $d$, McKean and Shatah  \cite{MS} answered positively to our question
 for $1+\frac{2}{d}<p<1+\frac{4}{d}.$   The case $p\ge 1+\frac{4}{d} $
 and $p<1+\frac{4}{d-1}$ for $d\ge 3$
  was  answered positively by  W. Strauss \cite{S81}.
  W. Strauss \cite{S74}   proved    that the  zero solution is the
  only
asymptotically free solution   when $1 < p \leq  1+\frac{2}{d}$ for
$d \geq 2,$ and when $1 < p \leq 2$ for $d = 1.$  This result was extended to the case $1 < p \leq 3$ and $d=1 $ by J. Barab \cite{B}, using an idea
of R. Glassey \cite{G73}. The exponent
  $p=1+\frac{2}{d}$ is    critical and particularly interesting.  The existence and the form of the
  scattering operator was obtained by Ozawa   \cite{O91} for $d=1$
 and by Ginibre and Ozawa  \cite{GO} for $d\ge 2$. The completeness of the
 scattering operator and the  decay estimate were obtained by Hayashi and Naumkin
 \cite{HN}.    Completeness of the
 scattering operator and    decay estimate  for all solutions, 
not only  for small ones,
 for $d=1$ and  $\lambda <0$, were obtained by Deift and Zhou
 \cite{DZ}. See also  \cite{DZ3,DZ2}  for  earlier references and
\cite{DM} for a simpler proof.  The result was extended to perturbations of   the defocusing cubic NLS    for $d=1$   in   \cite{DZ1}.
For    the focusing cubic NLS  for $d=1$, the pure radiation case, along with other cases reducible to the pure radiation one by means of Darboux transformations,     was treated in    \cite{DP}, proceeding along the lines of \cite{DZ}.

 Our goal in the present paper is to extend the result  of  McKean and Shatah  \cite{MS} to the case $V\neq 0$
 and $d=1$, which to our knowledge is open. For $V$ we assume the following hypothesis,  where we refer to Sect. \ref{sec:spectral} for the definition of the transmission coefficient $T(\tau )$.

\begin{itemize}
\item[(H)]  The potential    $V$ is a real valued Schwartz function such that for the spectrum  we have $\sigma (\triangle _V )=(-\infty ,0]$. Furthermore, $V$ is  generic, that is  the transmission coefficient $T(\tau )$ satisfies $T(0)=0$.

\end{itemize}
We denote by $ \Sigma _{s}$   the Hilbert space defined as the closure of $C_0^\infty({\mathbf R})$ functions with respect
to the norm
\begin{equation*}\label{eq:Sigma}
  \| u \| _{\Sigma _s}^2 :=  \| u \| _{H^s ({\mathbf R}  )}^2+  \| \, |x|^{s} u \| _{L^2({\mathbf R}  )}^2  .
\end{equation*}
Our main result is the following
\begin{theorem}\label{thm:dispersion} Assume that $V$  satisfies (H),
$s>1/2$ and $p>3 $.
Then there exist constants
$\epsilon _0>0$ and $C _0>0$ such that for    $\epsilon \in (0, \epsilon _0) $ and $ \| u (1)\| _{\Sigma _s} \leq \epsilon  $
the solution to \eqref{eq:PotNLS} satisfies  the decay inequality \eqref{eq:dispersion} for $t\ge 1$.
Furthermore  there exists $u_+\in L^2( {\mathbf R} )$ such that
\begin{equation}\label{eq:scattering} \lim _{t\to +\infty}
  \| u(t)   - e^{{\bf i}t\triangle } u_+\| _{L^2( {\mathbf R} )} =   0 .
\end{equation}
  \end{theorem}
The hypothesis $\sigma (\triangle _V )=(-\infty ,0]$ is necessary since
otherwise for any $s>1/2$ there are periodic solutions $u(t,x) =e^{{\bf i}\lambda t}\phi _\lambda (x)$
of arbitrarily small  $\Sigma _s$ norm.  The interesting case is for $ p\in (3,5)$ since the case $p\ge 5$ follows from \cite{GY,We1}.
The case $V=0$  is due to   \cite{MS}.

  If $\sigma (\triangle _V )=(-\infty ,0]$, the   existence of wave operators  intertwining
 $\triangle _V$ and $\triangle  $
 and of  Strichartz and dispersive   estimates  for $e^{{\bf i}t\triangle _V }$  is well known, see \cite{GY,We1,We2}. Such estimates are not   sufficient  to prove
 Theorem  \ref{thm:dispersion} even in the case $V=0$.

 The argument in  \cite{MS}   is based on the introduction of homogeneous
$\dot {\mathcal{H}}^k(t)$ norms,
defined substituting the standard derivative $\frac{\partial}{\partial x_j}$ with
operators $J_j(t)$, see Sect. \ref{sec:hom}. In \cite{MS} it is proved almost
invariance of these norms and, by a form of  the Sobolev embedding theorem,
the dispersion \eqref{eq:dispersion}. Such use of invariant norms goes back
to the work on the wave equation by
Klainerman, see for example \cite{Klainerman}.

 The  development of a theory of  invariant norms in the case
 of non translation invariant equations such as \eqref{eq:PotNLS}
 is an important technical problem. Here our main goal is
 to adapt the  framework of \cite{MS} for $d=1$ and  to introduce
appropriate  surrogates $|J_V(t)|^s  $ for the   operators $|J (t)|^s  $ see  Sect. \ref{sec:hom}.

The operators  $|J_V(t)|^s  $  are   used to define  homogeneous spaces $\dot {\mathcal{H}}  _V^s(t)$ which are then shown to be
almost invariant.

The argument is more complicated than in \cite{MS} because of
the presence of an additional commutator.  But we can   show that if $\triangle _V$
is generic, in the sense of Hypothesis (H), then the commutator can be treated
by a bootstrap argument.

Another complication is that the
 $|J_V(t)|^s  $ do not enjoy     Leibnitz rule  type properties  like
  $|J (t)|^s  $, which play a key role in \cite{MS}. Nonetheless,
  we are able to treat $|J_V(t)|^s  $   by switching from  $|J_V(t)|^s  $ to $|J (t)|^s  $, by using the   Leibnitz rule for   $|J (t)|^s  $,
  and by going back to $|J_V(t)|^s  $.

  In the  part of the argument on the  Leibnitz rule,
  an essential   role is played  by the
observation    that   $ \| \cdot \| _{\dot {\mathcal{H}} _V^s(t)} \approx
\| \cdot \| _{\dot {\mathcal{H}}  ^s(t)}$ with fixed constants independent of
$t$  when $0\le s <1/2$. The proof of this  equivalence   is based on Paley-Littlewood decompositions
associated to phase spaces both of
$\triangle $ and $\triangle _V$.
We are able to prove this equivalence when the  transmission
coefficient $T(\tau )$ is such that  either $T(0)=0$  (the generic case)
or $T(0)=1$. Notice incidentally that the inclusion of
this non generic case at least in this part of the paper is natural,
since the fact that $T(0)=1$ makes
 $\triangle _V$ more similar to $\triangle $  than the case when $T(0)=0$ (recall that $T(0)=1$ for $\triangle  $).

We introduce  now some of the notation  used later. Inequalities of type $A \lesssim B$   mean the existence of a constant $C>0$ so that $ A \leq C B.$ Similarly, $A \sim B$ means $A \lesssim B$ and $B \lesssim A$.
The standard scalar product in $ L^2= L^2({\mathbf R})$ will be denoted by $\langle.,.\rangle_{L^2}.$
We use the notation $L^p_x$ that means $ L^p({\mathbf R})$.  $L^p_t(X)$ stands for the $L^p$ norm of functions with values in Banach space $X.$ The homogeneous Sobolev space
$\dot{\mathcal{H}}^s({\mathbf R})$     (resp.  perturbed Sobolev space $\dot{\mathcal{H}}^s_V({\mathbf R})$)  for $s \geq 0$ is defined as the closure of $C_0^\infty({\mathbf R})$ functions with respect to the norm
$$\|(-\triangle )^{\frac s2} f \|_{L^2_x} \text{( resp . $  \|(-\triangle +V)^{\frac s2} f \|_{L^2_x} $). }$$  
 These norms are used in two   cases: functions depending only in $x$ and functions depending on both $t$ and $x$.

\section{Definition of $|J_V(t)|^{s}$}
\label{sec:hom}
In this section we assume $x\in {\mathbf R} ^d$ with $d$ a generic dimension and
we consider
\begin{equation*}\label{eq:LS}
 ({\bf i} \partial _t  + \triangle     ) u=
 0  .
\end{equation*}
Recall that the fundamental solution is given by
$e^{{\bf i}t\triangle} (x,y)= \frac{e^{{\bf i}\frac{(x-y)^2}{4t}}}
  { (4\pi {\bf i}t  ) ^{\frac{d}{2}}}$  for $t> 0$.
	
\noindent Consider the Fourier transform  $F$ and its inverse:
\begin{align}\label{eq:FT}
   & Ff(x)=
 ( 2\pi )^{-\frac{d}{2}}\int  _{{\mathbf R} ^d} e^{{\bf i}x\cdot y} f(y) dy, \\ \nonumber & F^{-1}f(x)=
 ( 2\pi )^{-\frac{d}{2}} \int _{{\mathbf R} ^d}  e^{-{\bf i}x\cdot y} f(y) dy.
\end{align}
We introduce also the dilation operator
$  D(t) \psi (x)
=  (2{\bf i}t )^{-\frac{d}{2}} \psi (\frac{x}{2t})$
and the multiplier operator $  M(t) \psi (x)
=  e^{\frac{{\bf i}x^2}{4t}}\psi (x).$
Then we have the following well known formula
\begin{equation*}\label{eq:FunSol1}e^{{\bf i}t\triangle} = M(t)
D(t)F^{-1} M(t).
\end{equation*}
Let $g(x)$ be a function and denote by $g(q)$ the multiplier operator
$g(q)\psi (x):= g(x)\psi (x)$.  We set  $p _j:={\bf i}\partial _{x_j}$ and $p=(p_1,..., p_d)$. More
generally, set $g(p) :=F^{-1} g(q) F $. The following
identity is well-known:
\begin{equation}\label{eq:identity}e^{{\bf i}t\triangle}
g(q)e^{-{\bf i}t\triangle}= M(t)g(2tp)   M(-t).
\end{equation}
for any $g(x)$. With an abuse of notation we will denote the
operator $g(q)$ by $g(x)$. Notice that we have
\begin{equation*}\label{eq:Comm}\begin{aligned}&  \left [
{\bf i} \partial _t  + \triangle, e^{{\bf i}t\triangle}
g(x)e^{-{\bf i}t\triangle} \right ] =\\&  e^{{\bf i}t\triangle}\left [ - \triangle ,g(x)\right ] e^{-{\bf i}t\triangle}
+ e^{{\bf i}t\triangle}\left [  \triangle ,g(x)\right ]
e^{-{\bf i}t\triangle}  =  0,
\end{aligned}
\end{equation*}
so obviously the same commutation rule holds for the r.h.s. of
\eqref{eq:identity}. In particular for $g(x)=x_j$ we get
on the r.h.s. of \eqref{eq:identity} the operators $J_j=2t\im
e^{\frac{{\bf i}x^2}{4t}} \partial _{x _j}e^{-\frac{{\bf i}x^2}{4t}}= 2t\im
\partial _{x _j}+{x _j}$ and we have $$ \left [
{\bf i} \partial _t  + \triangle , J _j\right ] =0.$$
We introduce for any $s\ge 0$ the following two operators:
\begin{eqnarray} \label{eq:Js} & |J (t)|^{s}  := M(t)
(-t^2\triangle )^{\frac{s}{2}} M(-t) \\&
\label{eq:JVs}|J_V(t)|^{s} := M(t)
(-t^2\triangle_V)^{\frac{s}{2}} M(-t).
\end{eqnarray}

\section{Commutative properties of  $|J_V(t)|^{s}$}

We  start the section by establishing some useful
commutator relations. In this section $x\in {\mathbf R} ^d$ with $d$ a generic dimension
and  $ M(t) = e^{{\bf i}|x|^2/4t}$.
\begin{lemma}\label{L1}
We have the following identities:
\begin{eqnarray}\label{M1}\nonumber
\left[{\bf i}\partial_t, M(t) \right]=\frac{x^2}{4t^2} M(t), \quad
\left[{\bf i}\partial_t, M(-t) \right]=-\frac{x^2}{4t^2} M(-t).
 \end{eqnarray}
\end{lemma}
\begin{proof}
A simple calculation gives
\begin{eqnarray}\label{M1-proof}\nonumber
{\bf i}\partial_t M(t)f -M(t){\bf i}\partial_t f=\left({\bf i}\partial_t M(t)
\right)f=\frac{x^2}{4t^2} M(t).
\end{eqnarray}
The second relation can be verified similarly.
\end{proof}

Furthermore, we shall prove the following:
\begin{lemma}\label{L2} We have:
\begin{eqnarray}\label{M2}
\nonumber \left[\triangle , M(t) \right]=M(t)\left(\frac{{\bf i}\, d }{2t}-\frac{x^2}{4t^2}+\frac{{\bf i}x\cdot  \nabla}{t} \right);\\ \nonumber
\left[\triangle , M(-t)
\right]=M(-t)\left(-\frac{{\bf i}\, d}{2t}-\frac{x^2}{4t^2}-\frac{{\bf i}x\cdot  \nabla}{t}
\right).
 \end{eqnarray}
\end{lemma}
\begin{proof}
For the first relation we have
\begin{equation}\label{M2-proof}\begin{aligned} &
  \left[\triangle , M(t) \right]f=f\triangle M(t)+2\nabla M(t)\cdot \nabla f \\ &
=M(t)\frac{{\bf i}\, d}{2t}f-M(t)\frac{x^2}{4t^2}f+M(t)\frac{{\bf i}x\cdot  \nabla
f}{t}.\end{aligned}\nonumber
\end{equation}
The   second relation  follows taking complex conjugates.
\end{proof}

>From Lemma \ref{L1} and Lemma \ref{L2} we get:
\begin{lemma}\label{L3}
The following commutator relations hold:
\begin{eqnarray}\label{M3}
\nonumber \left[{\bf i}\partial_t+\triangle ,
M(t) \right]=M(t)\left(\frac{{\bf i}d}{2t}+\frac{{\bf i}x \cdot\ \nabla }{t} \right),\\
\left[{\bf i}\partial_t+\triangle, M(-t)
\right]=M(-t)\left(-\frac{{\bf i}d}{2t}-\frac{x^2}{2t^2}-\frac{{\bf i}x\cdot\nabla}{t}
\right).\nonumber
 \end{eqnarray}
\end{lemma}
\begin{proof}
We shall check only the first relation, which follows
directly from above Lemmas   and
\begin{eqnarray}\label{M3-proof}
\left[{\bf i}\partial_t+ \triangle, M(t)
\right]=\left[{\bf i}\partial_t, M(t)
\right]+ \left[\triangle, M(t) \right] . \nonumber
\end{eqnarray}

\end{proof}

\begin{lemma}\label{L4} We have
\begin{eqnarray}\label{deltaV}
 \left[{\bf i}\partial_t+ \triangle_V, (-t^2\triangle_V)^{\frac{s}{2}}
 \right]=\frac{{\bf i}s}{t}(-t^2\triangle_V)^{\frac{s}{2}}.
 \end{eqnarray}

\end{lemma}
\begin{proof}
To prove \eqref{deltaV} we shall use the fact that
$(-\triangle_V)^{{s}/{2}}$ and $\triangle_V$ commute. Thus, we
have
\begin{eqnarray}\label{deltaV-proof}
\nonumber \left[{\bf i}\partial_t+ \triangle_V,
(-t^2\triangle_V)^{\frac{s}{2}}
\right]f\\
\nonumber =\left[{\bf i}\partial_t,(-t^2\triangle_V)^{\frac{s}{2}}
\right]f+ \left[\triangle_V,
(-t^2\triangle_V)^{\frac{s}{2}} \right]f\\
={\bf i}\partial_t\left((-t^2\triangle_V)^{\frac{s}{2}}\right)f=\frac{{\bf i}s}{t}(-t^2\triangle_V)^{\frac{s}{2}}f.\nonumber
\end{eqnarray}
\end{proof}

Now we are ready to establish the main commutative
property of the operator $|J_V(t)|^{s} $  with $s \geq 0$  defined in
\eqref{eq:JVs}.
\begin{proposition}\label{P1}
We have the relation:
\begin{equation}\label{JVs}
  \left[{\bf i}\partial_t+\triangle_V, |J_V(t)|^{s} \right]
 ={\bf i}t^{s-1}M(t) A(s)  M(-t)
\end{equation}
where
$$
 A(s):= { s} (- \triangle_V)^{\frac{s}{2}}+ \left[x \cdot \nabla ,
(- \triangle_V)^{\frac{s}{2}}\right] .
$$
\end{proposition}
\begin{proof}
The proof relies on Lemmas \ref{L1}--\ref{L4} and the following
commutator equalities:
\begin{eqnarray*}
\left[AB,C \right]=A\left[B,C \right]+\left[A,C \right]B,\\
\left[A,BC \right]=\left[A,B \right]C+B\left[A,C \right].
\end{eqnarray*}

Indeed, we have
\begin{equation}\label{JVs-proof1}\begin{aligned} &
  \left[{\bf i}\partial_t+ \triangle_V, |J_V(t)|^{s}\right]  =\left[{\bf i}\partial_t+ \triangle_V, M(t) (-t^2\triangle_V)^{\frac{s}{2}} M(-t) \right]\\ &
 =\left[{\bf i}\partial_t+ \triangle_V, M(t) \right](-t^2\triangle_V)^{\frac{s}{2}} M(-t) +M(t)\left[{\bf i}\partial_t+ \triangle_V,(-t^2\triangle_V)^{\frac{s}{2}} M(-t)\right]\\&=\frac{{\bf i}d}{2t}|J_V(t)|^{s}+\frac{{\bf i}}{t}M(t)x\cdot\nabla (-t^2\triangle_V)^{\frac{s}{2}}M(-t)\\
& +M(t)\left[{\bf i}\partial_t+ \triangle_V,(-t^2\triangle_V)^{\frac{s}{2}}\right] M(-t) + M(t)(-t^2\triangle_V)^{\frac{s}{2}}\left[{\bf i} \partial_t+ \triangle_V, M(-t)\right]\\&=\frac{{\bf i}d}{2t}|J_V(t)|^{s}+\frac{{\bf i}}{t}M(t)x\cdot\nabla (-t^2\triangle_V)^{\frac{s}{2}}M(-t)\\
&+\frac{{\bf i}s}{t}|J_V(t)|^{s}+M(t)
(-t^2\triangle_V)^{\frac{s}{2}}M(-t)\left(-\frac{{\bf i}d}{2t}-\frac{x^2}{2t^2}-\frac{{\bf i}x\cdot\nabla}{t}\right)\\
&=\frac{{\bf i}s}{t}|J_V(t)|^{s}+\frac{{\bf i}}{t}M(t)x\cdot\nabla (-t^2\triangle_V)^{\frac{s}{2}}M(-t)\\
&-\frac{{\bf i}}{t}M(t)
(-t^2\triangle_V)^{\frac{s}{2}} M(-t) x\cdot\nabla
-M(t)(-t^2\triangle_V)^{\frac{s}{2}}\frac{x^2}{2t^2}M(-t)\\
&=\frac{{\bf i}s}{t}|J_V(t)|^{s}+\frac{{\bf i}}{t}M(t)\left[x\cdot\nabla,
(-t^2\triangle_V)^{\frac{s}{2}}M(-t)\right]
-M(t)(-t^2\triangle_V)^{\frac{s}{2}}\frac{x^2}{2t^2}M(-t).
 \end{aligned}\nonumber
\end{equation}
Note that
\begin{eqnarray}\label{xnabla}
\nonumber\left[x\cdot\nabla,(-t^2\triangle_V)^{\frac{s}{2}}M(-t)\right]\\
\nonumber=\left[x\cdot\nabla,(-t^2\triangle_V)^{\frac{s}{2}}\right]M(-t)
+(-t^2\triangle_V)^{\frac{s}{2}}\left[x\cdot\nabla,M(-t)\right]\\
=\left[x\cdot\nabla,(-t^2\triangle_V)^{\frac{s}{2}}\right]M(-t)
-(-t^2\triangle_V)^{\frac{s}{2}}\frac{{\bf i}x^2}{2t}M(-t)
 \end{eqnarray}
and hence we get
\begin{eqnarray}\label{JVs-proof2}
\nonumber \left[{\bf i}\partial_t+ \triangle_V, |J_V(t)|^{s}\right]\\
 =\frac{{\bf i}s}{t}|J_V(t)|^{s}+\frac{{\bf i}}{t}M(t)\left[x\cdot\nabla, (-t^2\triangle_V)^{\frac{s}{2}}
\right]M(-t)\nonumber
 \end{eqnarray}
The proof of \eqref{JVs} is completed.
\end{proof}

In next lemma we shall assume $d=1$.

\begin{lemma}
\label{lem:boundA} Assume $d=1$ and
$A(s)$ be the operator that appears in \eqref{JVs} with $s<2$. Then for a fixed constant
   $C_s$ we have the inequality \begin{equation}\label{eq:boundA}\begin{aligned} &
  \| A(s)f\|  _{L^1_x}\le C_s\|  f\|  _{L^\infty_x}.
\end{aligned}\end{equation}
\end{lemma}
We postpone the proof of Lemma \ref{lem:boundA} to Sect. \ref{sec:boundA}.

\section{Spectral theory for $\triangle _V$}
\label{sec:spectral}
Since now on we shall always work in the space dimension $d=1$.\\
In this section we remind some classical material needed later.
Recall that the  Jost functions are  solutions $f_{
\pm } (x,\tau )=e^{\pm {\bf i}\tau x}m_{
\pm } (x,\tau )$ of $- \triangle _Vu=\tau ^2 u$ with
$$ \lim _{x\to +\infty }   {m_{ + } (x,\tau )}  =1 =
\lim _{x\to -\infty }  {m_{- } (x,\tau )}  . $$
We set  $x^+:=\max \{ 0,x \}$,   $x^-:=\max \{ 0,-x \}$  and $\langle x  \rangle =\sqrt{1+x^2}.$ We will denote by $L^{p,s}$ the space with  norm

\begin{equation}\label{eq:lps}\begin{aligned} &
  \|  u \|  _{L^{p,s}}= \| \langle x \rangle ^s  f\|  _{L^p_x}.
\end{aligned}\end{equation}

The following lemma  is well known.
\begin{lemma}
\label{lem:Jost} For $V\in \mathcal{S}({\mathbf R} )$ we have
  $m_\pm  \in C^{\infty} ({\mathbf R} ^2,{\mathbf C} )$.
    There exist  constants $C_1=C_1(\|   V    \|  _{L^{1,1}})$  and $C_2=C_2(\|   V    \|  _{L^{1,2}})$
such that:

   \begin{align}  \label{eq:kernel2}
 &  |m_\pm(x, \tau )-1|\le  C _1 \langle  x^{\mp }\rangle
\langle \tau \rangle ^{-1}\left | \int _x^{\pm \infty}\langle y \rangle |V(y)| dy \right |  \ ;
  \\&   \label{eq:tderm}   |  \partial _\tau   m_\pm(x, \tau ) | \le  C_2 (1+ x^2 ) .
     \end{align}
    \end{lemma}
  See Lemma 1 p. 130 \cite{DT}. The regularity  follows iterating the argument.

The transmission coefficient
$T(\tau )$ and the reflection  coefficients
$R_\pm (\tau )$   are defined by  the formula
   \begin{equation}  \label{eq:kernel35} \begin{aligned} &    T(\tau )m_\mp (x ,\tau )= R_\pm (\tau )e^{\pm 2{\bf i}\tau x }m_\pm (x,\tau )+   m_\pm (x,-\tau ).
\end{aligned}
\end{equation}
 From \cite{DT}  and from \cite{We2} we have the following lemma.
 \begin{lemma}
\label{lem:TRcoeff} For $V\in \mathcal{S}({\mathbf R} )$ we have
  $T, R_\pm   \in C^{\infty} ({\mathbf R}  )$. Moreover:
    \begin{align} &  \label{eq:TRcoeff0}
		 |T(  \tau )-1| +| R_\pm (  \tau )  |\le C \langle \tau  \rangle ^{ -1}  \text{ for  $C=C(\|   V    \|  _{L^{1,1}} )$;} \\&
		\label{eq:TRcoeff}
     |T(\tau )|^2+ |R_{\pm }(\tau )|^2=1; \\&   \label{eq:TRcoeff1}
		 |   \frac d {d\tau }
		T(  \tau ) | +|  \frac d {d\tau } R_\pm (  \tau )  |\le C    \text{ for  $C=C(\|   V    \|  _{L^{1,3}})$} .
    \end{align}
\end{lemma}
In particular,  \eqref{eq:TRcoeff} and\eqref{eq:TRcoeff1}  follow from Sect.3  \cite{DT}
and  \eqref{eq:TRcoeff0}
follows from  Theorem 2.3  \cite{We2}.

\noindent Set now $\Psi  (x,\tau )=  T(  \tau ) f_+(x,\tau )$ for $\tau \ge 0$  and
$\Psi  (x,\tau )=  T(  -\tau ) f_-(x,-\tau )$ for $\tau \le 0$. Then the
distorted Fourier transform associated to $\triangle _V$ is defined by

\begin{equation}\label{eq:distFT}\begin{aligned} & F_Vf(\tau )=
 ( 2\pi )^{-\frac{1}{2}}\int  _{{\mathbf R}  }\Psi  (x,\tau ) f(x) dx
 \end{aligned}
\end{equation}
and we have   the inverse   formula

\begin{equation}\label{eq:distFTinv}\begin{aligned} &    f(x)=
 ( 2\pi )^{-\frac{1}{2}} \int _{{\mathbf R}  } \overline{ \Psi  (x,\tau )} F_Vf(\tau ) d\tau .
 \end{aligned}
\end{equation}
Our first application of this theory is the following lemma.

\begin{lemma}
  \label{lem:Sob} Let $V\in \mathcal{S}({\mathbf R}) $ and $\sigma (\triangle _V )=(-\infty ,0]$,
then for any $s> 1/2$
  there exists a fixed $C$ such that:
  \begin{equation} \label{eq:interp}  \begin{aligned} & \| f\|_{L^\infty_x}   \leq C \| f
\|^{1-\frac{1}{ 2s}}_{L^2_x} \| f\|^{ \frac{1}{2 s}}_{\dot H ^s_V}.\end{aligned}
\end{equation}
  \end{lemma}
\begin{proof}

We claim  that $\| f\|_{L^\infty_x}\le  c_0\| F_V {f}\|_{L^1_x}$
for a fixed $c_0=c_0(V)$. Assuming the claim we have:
\begin{equation}   \begin{aligned} & \| F_V {f}\|_{L^1_x}   \le
\| F_V {f}\|_{L^2( |\xi |\le \kappa  )}  \sqrt{2} \kappa ^{\frac{1}{2}} +
\| \, |\xi |^{s}  F_V {f}\|_{L^2( |\xi |\ge \kappa  )} \| \, |\xi |^{-s}   \|_{L^2( |\xi |\ge \kappa  )}
\\& \le \sqrt{2} \kappa ^{\frac{1}{2}} \,
\|  {f}\|_{L^2_x }   +C_s  \kappa ^{\frac{1}{2}- s} \|  {f}\|_{\dot H^s_V }    \text{ with } C_s:= \sqrt{\frac{2}{ { 2s-1}}}.\end{aligned}\nonumber
\end{equation}
 For $\kappa = \left ( 2^{-\frac{1}{2}} C_s\|  {f}\|_{\dot H^s _V} \right )^{\frac{1}{ s}}
\|  {f}\|_{L^2_x }^{-\frac{1}{ s}}$ the last two terms are equal
 and we get \eqref{eq:interp}.

 We now prove  $\| f\|_{L^\infty_x}\le  c_0\| F_V {f}\|_{L^1_x}$. By \eqref{eq:distFTinv} it suffices to prove $|\Psi (x,\tau )|\le C_0$ for fixed $C_0$. It is not restrictive to assume $x>0$. Then
  for $\tau \ge 0$ we get the bound by
  $\Psi  (x,\tau )=  T(  \tau ) f_+(x,\tau )$ and Lemmas \ref{lem:Jost}  and \ref{lem:TRcoeff}.
  Similarly for   $\tau < 0$  we get a similar bound  by

  \begin{equation*}   \begin{aligned} &   \Psi  (x,\tau )=  T(  -\tau ) f_-(x,-\tau )  = R_+ (-\tau )f_+ (x,-\tau )+   f_+ (x, \tau ).
\end{aligned}
\end{equation*}

\end{proof}

Consider   now a function $u(t,x)$.
By Lemma   \ref{lem:Sob}
we have  for $s>1/2$:
\begin{equation} \label{eq:decay1}  \begin{aligned}   \|  u(t, \cdot)\|_{L^\infty_x} &
\leq C\|M(-t)u(t, \cdot)\| ^{1-\frac{1}{2 s}}_{L^2_x}
\| M(-t)u(t, \cdot)\|^{ \frac{1}{2 s}}_{\dot H ^s_V} \\& =
\frac{C}{\sqrt{t}}\| u(t, \cdot)\| ^{1-\frac{1}{2 s}}_{L^2_x} \|
|J _V(t) |^su(t, \cdot)\|^{ \frac{1}{2 s}}_{L^2_x}.\end{aligned}
\end{equation}

\section{Proof of Theorem \ref{thm:dispersion}}

 Using the notation of  Proposition \ref{P1}
 we have
 the following equation
\begin{equation}\label{eq:NLSJu}
 ({\bf i} \partial _t  + \triangle _V )|J_V |^{s} u-{\bf i}t^{s-1}M(t) A(s)  M(-t)u+
\lambda |J_V |^{s} F=0  ,
\end{equation}
with $F= |u| ^{p-1} u $.
Let $0<s<2$. Then by Strichartz estimates which follow by \cite{We2},
there are   fixed $C_s'$ and $C$  s.t.

\begin{equation} \label{eq:Strich1}\begin{aligned} &
   \| |J_V |^{s} u \|_{L^\infty ((1,T),L^2_x)}\le C \| |J_V |^{s} (1) u \|_{ L^2_x}
   \\& +
   C_s'\| t^{s-1} A(s)M(-t)u\|_{L^{\frac{4}{3}} ((1,T),L^1_x)}   + C \| |J_V |^{s} F\|_{L^{1}( (1,T),L^2_x)}.
\end{aligned}\end{equation}
By combining Lemma \ref{lem:boundA}, \eqref{eq:decay1} and conservation of
charge we get for every $\delta>0$ a constant $M(\delta)$ such that
\begin{equation*} \begin{aligned} &
   \| t^{s-1} A(s)M(-t)u\|_{L_t ^{\frac{4}{3}} L^1_x}  \le C_s \| t^{s-1} \|  u\|_{  L^\infty _x} \|_{L_t ^{\frac{4}{3}}  } \\& \le D_s \| t^{s- \frac{3}{2}} \|_{L_t ^{\frac{4}{3}}  }
   \| u(1)\| ^{1-\frac{1}{2 s}}_{L^2_x} \| |J _V  |^su \|^{ \frac{1}{2 s}}_{L_t ^{\infty} L^2_x}\le M (\delta )\| u(1)\|  _{L^2_x} +\delta  \| |J _V  |^su \| _{L_t ^{\infty} L^2_x},
   \end{aligned}\end{equation*}
where we have considered $s<\frac{3}{4}$
so that $t^{s- \frac{3}{2}}\in L^{\frac{4}{3}} (1,\infty).$
Inserting this estimate in \eqref{eq:Strich1} we conclude
\begin{equation} \label{eq:Strich2}\begin{aligned} &
   \| |J_V |^{s} u \|_{L^\infty ((1,T),L^2_x)}\le C \| |J_V |^{s}  u (1)\|_{ L^2_x} \\& +C_s \| u(1)\|  _{L^2_x} + C_s \| |J_V |^{s} F\|_{L ^{1} ((1,T),L^2_x)}.
\end{aligned}\nonumber\end{equation}

We shall use the following result.

\begin{lemma}
\label{l:eqjsv}
We have
\begin{equation}\label{eqss1}
   \| |J _V|^s f \|_{L^2_x} \sim  \|J^s f \|_{L^2_x} \text{ for $0 \leq s < 1/2$ }.
\end{equation}
For $s \in (1/2,1)$ and any $\varepsilon \in (0,1/2)$ we have:
\begin{eqnarray}
\label{eq.injj} &\| |J _V|^s f \|_{L^2_x} \leq C t^{s+\varepsilon - \frac 1 2}   \left( \| |J|^{\frac 1 2-\varepsilon}f\|_{L^2_x} + \| |J  |^s f \|_{L^2_x}\right);  \\
  \label{eq.injj1} &\| |J |^s f \|_{L^2_x} \leq C t^{s+\varepsilon - \frac 1 2}   \left( \| |J_V|^{\frac 1 2-\varepsilon}f\|_{L^2_x} + \| |J _V |^s f \|_{L^2_x}\right) .
\end{eqnarray}

\end{lemma}

\begin{proof} \eqref{eqss1} is a simple consequence of  Corollary \ref{l:ineperteq} in the next section  which states
\begin{equation}\label{eq:normeq}  \|(-\triangle)^{\frac s2} f \|_{L^2_x} \sim
\|(-\triangle+V)^{\frac s2} f \|_{L^2_x}
\text{ for $0 < s < 1/2 $.}\end{equation}

\noindent To prove \eqref{eq.injj} (resp. \eqref{eq.injj1})
we use
\begin{equation*} \begin{aligned} &
\| \sqrt{-\triangle _V} f \|^2_{L^2_x}   \leq
 \| \sqrt{-\triangle}  f \|^2_{L^2_x} + \|  V f^2\| _{L^1_x}\\&
\|  V f^2\| _{L^1_x} \le  \|  V \| _{L^{p'}_x}  \| f \| _{L^{2p}_x} ^{2}
 \le C \|(-\triangle)^{\frac{1}{4} -\frac{\delta}{2}} f \|^2_{L^2_x}
\text{ for } \frac{1}{2p} = \frac{1}{2 } -
(\frac{1}{2}- \delta ) = \delta \end{aligned}
\end{equation*}
(resp. the inequalities with $ \triangle _V$ and $\triangle$ interchanged:
this will use also \eqref{eq:normeq}).
 We thus obtain
$$ \|\sqrt{-\triangle _V} f \|_{L^2_x} \leq C \|(-\triangle)^{\frac{1}{4} -\frac{\delta}{2} } \left( 1 + (-\triangle)^{\frac{1}{4} +\frac{\delta}{2} }\right) f \|_{L^2_x}$$
(resp. the inequality with $ \triangle _V$ and $\triangle$ interchanged).
Interpolation   with \eqref{eq:normeq}  for $s= {1}/{2} - {\delta} $
 yields
$$ \|(-\triangle _V)^{\frac{s}{2} } f \|_{L^2_x} \leq C \|(-\triangle)^{\frac{1}{4} -\frac{\delta}{2}} \left( 1 + (-\triangle )^{ \frac{s}{2}- \frac{1}{4}+\frac{\delta}{2} }\right) f \|_{L^2_x}$$
$$ \leq C \left( \|(-\triangle)^{\frac{1}{4} -\frac{\delta}{2}} f \|_{L^2_x} + \| (-\triangle)^{\frac s2} f \|_{L^2_x}  \right)$$
(resp. the inequality with $ \triangle _V$ and $\triangle$ interchanged).
Multiplying this estimate by $t^s$
and using again the fact that $M(t)$ is $L^2_x$ bounded operator, we see that
 $$ \| |J_V|^s f \|_{L^2_x} \leq C \left( t^{s-\frac 12+\delta} \| |J|^{\frac 12-\delta}
f \|_{L^2_x} + \| |J|^{s} f \|_{L^2_x}   \right)$$
 and  for $\varepsilon = \delta  $  we get \eqref{eq.injj}
 (resp.
  \eqref{eq.injj1}).

\end{proof}
\noindent By Lemma \ref{l:eqjsv} we get
\begin{equation} \label{eq:Strich3}\begin{aligned} &
   \| |J_V |^{s} u \|_{L^\infty ((1,T),L^2_x)}\le C_s \|     u(1) \|_{ \Sigma_s}  + C_s \| |J_V |^{s} F\|_{L ^{1}( (1,T),L^2_x)},
\end{aligned}\nonumber\end{equation}
since
$$  \| |J |^{s}  u (1)\|_{ L^2_x} \leq C \|u(1)\|_{\Sigma_s}.$$
 If we can show that for a fixed $C$ for all $T$
\begin{equation}\label{eq:decay2}\begin{aligned} &\| |J_V |^{s} u \|_{L^\infty ((1,T),L^2_x)}\le C
\|u(1)\|_{\Sigma_s},   \end{aligned}
\end{equation}
then by \eqref{eq:decay1} this will yield \eqref{eq:dispersion}.
Then scattering \eqref{eq:scattering} will follow
from  \eqref{eq:dispersion} by a standard argument which we do not repeat.
\noindent  By combining Lemma \ref{l:eqjsv} with Lemma 2.3 in \cite{HN}, which states that $$\| |J|^\gamma
(|u| ^{p-1} u)\|_{ L^{2}_x  }\le C \| u\|_{ L^{\infty}_x  }^{p-1}\| |J|^\gamma   u\|_{ L^{2}_x  } \text{ for $0\le \gamma <2$ and $p\ge 3$}$$ we have
\begin{equation*}\label{mj2in3}\begin{aligned} &
   \| |J_V |^{s}  (|u| ^{p-1} u)\|_{L ^{1}  ((1,t),L^2_x) }\le
 \\&   C   \left \|  \langle t'\rangle^{s+\varepsilon - \frac 12} ( \| |J|^{\frac 12-\varepsilon}(|u| ^{p-1} u)\|_{   L^{2}_x } + \|J^s (|u| ^{p-1} u)\|_{ L^{2}_x }) \right \| _{ L^1(   1,t)} \\&
 \le C '     \left \|  \langle t'\rangle^{s+\varepsilon - \frac 12} \| u\|_{ L^{\infty}_x }^{p-1}( \| |J |^{\frac 12-\varepsilon}  u\|_{      L^{2}_x } + \| |J|^s   u\|_{ L^{2}_x }) \right \| _{ L^1 (   1,t)}  \\&
 \le C '     \left \|  \langle t'\rangle^{s+\varepsilon - \frac 12} \| u\|_{ L^{\infty}_x }^{p-1}( \| |J_V|^{\frac 12-\varepsilon}  u\|_{      L^{2}_x } + \| |J|^s   u\|_{ L^{2}_x }) \right \| _{ L^1(   1,t)}
\end{aligned}
\end{equation*}
Again by Lemma \ref{l:eqjsv} we can continue the estimate as follows
\begin{equation*}\label{mj2in3bis}\begin{aligned} &
... \le C '     \left \|  \langle t'\rangle^{2s+2\varepsilon - 1} \| u\|_{ L^{\infty}_x }^{p-1}( \| |J_V|^{\frac 12-\varepsilon}  u\|_{      L^{2}_x } + \| |J_V|^s   u\|_{ L^{2}_x }) \right \| _{ L^1(   1,t)}\le\\& C'\int _1^t  \langle t'\rangle
^{2(s+\varepsilon)   -\frac{p +1}{2}}  (\| u\|_{ L^{2}_x }^{2s-1} \| |J_V|^s    u\|_{ L^{2}_x })^{\frac{p-1}{2s}}( \| |J_V|^{\frac 1 2-\varepsilon}   u\|_{      L^{2}_x } + \| |J_V|^s    u\|_{ L^{2}_x })  dt'\end{aligned}
\end{equation*}
where in the last line we used \eqref{eq:decay1}.

 \noindent
Since $p>3$ we can choose $s>1/2 $ and $\varepsilon >0$ such that $  \frac{p+1 }{2} -2s-2\varepsilon >1$. Then
\begin{equation*} \begin{aligned} &
   \| |J_V |^{s}  (|u| ^{p-1} u)\|_{L ^{1}  _t L^2_x }\\& \le
  C_s  \| u (1)\|_{ L^{2}_x }^{(p-1) \frac{2s-1}{2s}} \| |J_V|^s    u\|_{ L ^{\infty}  _t L^2_x }^{\frac{p-1}{2s}}( \| |J_V|^{\frac 1 2-\varepsilon}   u\|_{     L ^{\infty} _t L^2_x } + \| |J_V|^s    u\|_{L ^{\infty}  _t L^2_x })  \end{aligned}
\end{equation*}
 on any  interval $(1,t)$
 a constant $C_s$ independent of $t$. Notice that the norm $\| |J_V|^{\frac 1 2-\varepsilon}   u\|_{     L ^{\infty} _t L^2_x }$ can be bounded in terms of the other norms  using interpolation, hence the proof of \eqref{eq:decay2}
follows by a standard continuity argument, provided that we fix the constant $\epsilon _0>0$ in
the statement  of Theorem \ref{thm:dispersion} sufficiently small.

\section{Equivalence of homogeneous Sobolev norms}

Along this section the functions $m_\pm (x, \tau)$, $f_\pm(x, \tau)$, $T(\tau)$ and $R_\pm (\tau)$
are the ones defined in Sect. \ref{sec:spectral}. Also the norm $\|V\|_{L^{p,q}}$
is the one defined in
the same section.
We consider for an appropriate cutoff  $\varphi \in C^\infty _0({\mathbf R} _+, [0,1])$
a Paley-Littlewood partition of unity
$$ 1 = \sum_{j \in {\mathbf Z}}  \varphi \left(t 2^{- j}\right) , t > 0.$$
Then for any $s\in {\mathbf R} $ we have

\begin{equation}\label{eq:DistSob 1} \begin{aligned}
   \|(-\triangle _V)^{\frac{s}{2}  } f\|^2_{L^2}
&\sim \sum_{j \in {\mathbf Z}} 2^{2 js} \langle \varphi \left( 2^{-j}\sqrt{-\triangle _V}
\right)  f, f \rangle_{L^2_x}\\& \sim \sum_{j \in {\mathbf Z}} 2^{2js} \| \varphi \left(
2^{-j} \sqrt{-\triangle _V} \right)  f\|^2_{L^2_x}.
\end{aligned}\nonumber\end{equation}
 We have the following result.

\begin{lemma}
\label{l:diadhigh} Let   $V$ be a real valued Schwartz function such that  $\sigma (\triangle _V )=(-\infty ,0]$ and $T(0)$ is either equal to 0
or to 1.  Then for  any pair of integer numbers $j,k \in {\mathbf Z} $ with $k \leq j$
 and for any  $f\in \mathcal{S}( \mathbb{{R}}),$ such that
 \begin{equation}\label{eq:suppf}
 {\rm supp} \widehat{f}(\xi) \subseteq \{  |\xi| \sim 2^k \},
 \end{equation}
  the following inequality holds for $C_V =C(\| V \| _{L^{1,3}})$:

\begin{equation}\label{eq:quasidiag}
\langle \varphi \left({2 ^{-j}}  {\sqrt{-\triangle _V}} \right )   f , f \rangle_{L^2_x}
\leq C_V  2^{-|k-j|} \|f\|^2_{L^2_x}.
\end{equation}

\end{lemma}

\begin{proof}
For $\varphi (|\tau |):= \tau ^2\psi (|\tau |)$ we have

\begin{equation*}    \begin{aligned} & \langle \varphi
\left( {2 ^{-j}}  {\sqrt{-\triangle _V}} \right)
f , f \rangle_{L^2_x} =
 A_j(f) + B_j(f) \\&  A_j(f)  := -2^{-2j} \langle \psi
\left( {2 ^{-j}}  {\sqrt{-\triangle _V}} \right)   f , \partial_x^2f   \rangle_{L^2_x}
   \\&  B_j (f) :=2^{-2j} \langle \psi \left( {2 ^{-j}}
{\sqrt{-\triangle _V}}\right)   f , V f   \rangle_{L^2_x} .\end{aligned}
\end{equation*}
It is  straightforward that

 \begin{equation}  \label{eq:aj1} \begin{aligned}
 & |A_j(f)| = 2^{-2j} | \langle \psi  \left( {2 ^{-j}}
{\sqrt{-\triangle _V}} \right)   f , \partial_x^2f \rangle_{L^2_x}|\\& \le 2^{-2j}
 \|  \psi  \left( {2 ^{-j}}  {\sqrt{-\triangle _V}} \right)   f \| _{L^2_x}
 \|   \partial_x^2f \| _{L^2_x} \le C 2^{2k-2j}\|    f \| _{L^2_x}^2.
     \end{aligned} \end{equation}
Notice that this constant $C$ depends on the cutoff $\varphi$
but not on $V$. This follows from the fact that the distorted  Fourier transform
\eqref{eq:distFT} is an isometry.
Next lemma in conjunction with \eqref{eq:aj1} will complete the proof
of Lemma \ref{l:diadhigh}.
\begin{lemma}
\label{l:diadhigh2} Assume the hypothesis of Lemma \ref{l:diadhigh}.
Then there exists a fixed  for $C  =C(\| V \| _{L^{1,3}})$ such that
 $ |B_j(f)| \leq C 2^{-|k-j|} \|f\|^2_{L^2_x} $.
\end{lemma}\begin{proof}
The first step in the proof is the following representation formula:

\begin{lemma}
\label{lem:repr}   We have

\begin{equation}\label{eq:b1}  \begin{aligned} &
   ( \psi  (2^{- j}\sqrt{ -\triangle _V}) f) (x)  =  - \frac  1 {2 \pi}  \int _{\mathbf R}    d\tau \psi (2^{- j}  \tau  )   \\&  \times\big [ T( \tau ) m_+(x, \tau )   \int _{y<x}m_-(y, \tau )   e^{{\bf i}\tau (x-y)} f(y)dy   \\&  +   T(- \tau ) m_-(x, -\tau )  \int _{y>x} m_+(y, -\tau )  e^{{\bf i}\tau (x-y)}  f(y)dy\big ] .
\end{aligned}
\end{equation}
\end{lemma}\begin{proof}
We recall the Limiting Absorption Principle
 \begin{equation} \label{eq:absprin1}  \begin{aligned}
 &    g( -\triangle _V) (x,y)=  \int _0^\infty    g( \lambda ) E _{a.c.}(d\lambda ) (x,y)   \\&    E _{a.c.}(d\lambda ) (x,y )=\frac  1 {2\pi \im}\left [ R^+_{ -\triangle _V}(x,y, \lambda  )- R^-_{ -\triangle _V}(x,y, \lambda  )\right ] d\lambda
\end{aligned}
\end{equation}
where  for $ \lambda  >0$ and $x<y$ (for  $x>y$ exchange $x$ and $y$ in the r.h.s.)
\begin{equation} \label{eq:absprin2}  \begin{aligned}
 &      R^\pm _{ -\triangle _V}(x,y, \lambda    )  =\frac{ f_- (x, \pm \sqrt{ \lambda   })   f_+ (y, \pm \sqrt{ \lambda   })   }
{w( \pm \sqrt{ \lambda    } )}
\end{aligned}
\end{equation}
for  the Wronskian
\begin{equation} \label{eq:wronskian}  \begin{aligned}
 &      w( {\tau}):= (\partial _xf_{
+ }) (x, {\tau} )  f_{
- } (x, {\tau}) - f_{
+ } (x, {\tau} )  \partial _xf_{
- } (x, {\tau})  .
\end{aligned}
\end{equation}
Then  for $x<y$  (for  $x>y$ exchange $x$ and $y$ in the r.h.s.)
\begin{equation*} \label{eq:absprin3}  \begin{aligned}
      g( -\triangle _V) (x,y)&=      \int _0^\infty   \tau    g( \tau    ^2)
\left [ \frac{ f_- (x,     \tau   )   f_+ (y,    \tau    )   }
{w(     \tau    )} -\frac{ f_- (x, -  \tau )   f_+ (y, -  \tau )   }
{w( -   \tau     )} \right ] \frac{d\tau  }{ \pi \im} \\& =  - \frac  1 {2 \pi}    \int _\R   T(\tau    ) g( \tau    ^2)
   f_- (x,     \tau  )   f_+ (y,    \tau    )     d\tau  ,
\end{aligned}
\end{equation*}
where we used the formula  $\frac 1{T( \tau    )}=\frac{w(    {\tau   } )} {2{\bf i}\tau  }$, see p. 144
\cite{DT}.    Therefore,  making   also a change of variable,
\begin{equation} \label{eq:absprin4}  \begin{aligned}
    g( -\triangle _V) f(x) =          - \frac  1 {2 \pi}    \int _{\mathbf R}   d\tau   \   g(& \tau     ^2)
 \big [ T( \tau    )  f_+ (x,    {\tau   }) \int _{-\infty } ^x    f_- (y,   {\tau   })   f(y) dy \\& +
T(- \tau    )  f_- (x,    -{\tau   }) \int _{x} ^\infty    f_+ (y,  - {\tau   })   f(y) dy  \big ] .
\end{aligned}
\end{equation}
For $g( \lambda     )= \psi (2^{- j}  \sqrt{ \lambda    } )$  and $ f_\pm  (x,   {\xi  })=e^{\pm {\bf i}x \xi }m_\pm  (x,   {\xi  }) $ we get Lemma  \ref{lem:repr}.
\end{proof}

We continue with the proof of Lemma \ref{l:diadhigh2}   by writing
   $$ B_j(f) =  B_j^{(1)}(f)+ B_j^{(2)}(f)$$  with

\begin{equation}\label{eq:b11}  \begin{aligned}
   &  B_j^{(1)}(f)  :=  - \frac  1 {2 \pi}2^{- 2j} \int _{{\mathbf R}}  dx
V(x) \overline{{f}(x) }  \\& \times\int _{\mathbf R}    d\tau \psi(2^{- j}  \tau  ) \big [ T( \tau ) m_+(x, \tau )   \int _{y<x}(m_-(y, \tau )-1)   e^{{\bf i}\tau (x-y)} f(y)dy  \\& +   T(- \tau ) m_-(x, -\tau )  \int _{y>x} (m_+(y, -\tau ) -1)
e^{{\bf i}\tau (x-y)}  f(y)dy\big ] \,
\end{aligned}
\end{equation}
and
\begin{equation}\label{eq:b12}  \begin{aligned}
    B_j^{(2)}(f) &:=  - \frac  1 {2 \pi}2^{- 2j} \int _{{\mathbf R}}  dx  V(x) \overline{{f}(x) }  \\& \times\int _{\mathbf R}    d\tau \psi(2^{- j}  \tau  ) \big [ T( \tau ) m_+(x, \tau )   \int _{y<x}    e^{{\bf i}\tau (x-y)} f(y)dy  \\& +   T(- \tau ) m_-(x, -\tau )  \int _{y>x}   e^{{\bf i}\tau (x-y)}  f(y)dy\big ] \ .
\end{aligned}
\end{equation}

\begin{lemma}
\label{lem:split1} Assume that $f$, $j$ and $k$ are as in Lemma \ref{l:diadhigh}.
Let  $V$ be a real valued Schwartz function such that  $\sigma (\triangle _V )=(-\infty ,0]$.
We do not impose other hypotheses on $V$.
Then,  for fixed $C=C(\| V \| _{L^{1,3}})$,  we have  $ |B_j^{(1)}(f)|
\le C 2^{k-  j}\| f\| _{L^2_x}^2$.
\end{lemma}
\begin{proof}
 The inequality  follows  from the following ones:
\begin{equation}\label{eq:b111}  \begin{aligned}
   &  |B_j^{(1)}(f)| \le C  2^{-  j} \|  \langle x \rangle ^3V\| _{L^1_x}
      \| f\| _{L^\infty_x}^2\le C'   2^{k-  j}\| f\| _{L^2_x}^2,
\end{aligned}
\end{equation}
with $C=C(\| V \| _{L^{1,3}})$, and where
we used  Bernstein inequality
\begin{equation}\label{eq:bern}  \begin{aligned}
   &  \|f\|_{L^\infty_x}\lesssim  2^\frac k2\|f\|_{L^2_x} .
\end{aligned}
\end{equation}
To prove the first inequality   in \eqref{eq:b111},
observe that the second line of \eqref{eq:b11}  can be bounded
 by   $ C \langle x \rangle ^3 \| f\| _{L^\infty_x}$  with $C=C(\| V \| _{L^{1,1}})$
using the following estimates, which follow from \eqref{eq:kernel2}:
\begin{equation*}
       \int _{-\infty } ^{x}|m_-(y, \tau )-1|     |f(y)| dy
\end{equation*}  \begin{equation*}
\lesssim  \| f\| _{L^\infty_x}
( \int _{-\infty } ^{x\wedge 0} \langle y \rangle ^{-2} dy +
   \int _0 ^{x\vee 0} \langle y \rangle   dy ) \lesssim \langle x \rangle ^2\| f\| _{L^\infty_x}  \, ,
\end{equation*}
and
\begin{equation*}
|m_+(x, \tau ) | \lesssim  \langle x \rangle .
\end{equation*}

\noindent Proceeding as above   the third line of \eqref{eq:b11}   can be bounded
by     $ C \langle x \rangle ^3 \| f\| _{L^\infty_x}$  with $C=C(\| V \| _{L^{1,1}})$ using  estimates like
\begin{equation*}
       \int ^{ \infty } _{x}|m_+(y, -\tau )-1|     |f(y)| dy     \lesssim
\end{equation*} \begin{equation*}
\| f
\| _{L^\infty_x} ( \int ^{ \infty } _{x\vee    0} \langle y \rangle ^{-2} dy +
   \int ^0 _{x\wedge  0} \langle y \rangle   dy )
\lesssim \langle x \rangle ^2\| f\| _{L^\infty_x}
\end{equation*}
and
\begin{equation*}
  |m_-(x, \tau ) |\lesssim  \langle x \rangle .
\end{equation*}
  This proves \eqref{eq:b111} and so also Lemma  \ref{lem:split1}.
\end{proof}

\begin{lemma}
\label{lem:split2} In addition to the hypotheseis of Lemma  \ref{lem:split1} let us assume now that
either $T(0)=0$ or $T(0)=1$.
Then    we have  $ |B_j^{(2)}(f)|  \le C 2^{k-  j}\| f\| _{L^2_x}^2$  for fixed    $C=C(\| V \| _{L^{1,3}})$.
\end{lemma}
\begin{proof}
We use \eqref{eq:kernel35} and substitute  \begin{equation}  \label{ker35bis}    \begin{aligned} &    T(-\tau )m_- (x ,-\tau )= R_+ (-\tau )e^{- 2{\bf i}\tau x }m_+ (x,-\tau )+   m_+ (x, \tau ).
\end{aligned}
\end{equation}
We then write

\begin{equation*}  \begin{aligned} &
    B_j^{(2)}(f)  =- \frac  1 {2 \pi} 2^{- 2j} \int _{{\mathbf R}}  dx  V(x)\overline{f(x)}  \\& \times\int _{\mathbf R}    d\tau \psi(2^{- j}  \tau  ) \big [ T( \tau ) m_+(x, \tau )   \int _{y<x}    e^{{\bf i}\tau (x-y)} f(y)dy\\& +  m_+(x, \tau )   \int _{y>x}   e^{{\bf i}\tau (x-y)}  f(y)dy  \\& +  R_+ (-\tau ) m_+ (x,-\tau )  \int _{y>x}   e^{-{\bf i}\tau (x+y)}  f(y)dy\big ] .
\end{aligned}
\end{equation*}
Notice that Lemma \ref{l:diadhigh}  is elementary for $| k-j|\le \kappa _0$ for any preassigned
$\kappa _0>1$. So we will focus only on the case $k-j> \kappa _0$ with a  fixed sufficiently large $\kappa _0 $.  We write
\begin{equation} \label{eq:cancel}       \begin{aligned} &  \psi(2^{- j}  \tau  )      \int _{y>x}   e^{{\bf i}\tau (x-y)}  f(y)dy   = \psi(2^{- j}  \tau  )  e^{{\bf i}\tau x}   \overbrace{\int _{{\mathbf R}}   e^{-{\bf i}\tau y }  f(y)dy}^{\sqrt{2\pi} \widehat{f}(-\tau )}  \\& -\psi(2^{- j}  \tau  )       \int _{y<x}   e^{{\bf i}\tau (x-y)}  f(y)dy = -\psi(2^{- j}  \tau  )       \int _{y<x}   e^{{\bf i}\tau (x-y)}  f(y)dy,
\end{aligned}
\end{equation}
because  $\psi(2^{- j}  \tau  ) \widehat{f}(-\tau )\equiv 0  $ for  $|j-k|> \kappa _0$  and $\kappa _0$   sufficiently large.

\noindent By  \eqref{eq:cancel} we can write

\begin{equation*}  \begin{aligned} &
    B_j^{(2)}(f)  = - \frac  1 {2 \pi}2^{- 2j} \int _{{\mathbf R}}  dx  V(x)\overline{f(x)}  \\& \times\int _{\mathbf R}    d\tau \psi(2^{- j}  \tau  ) \big [ (T( \tau ) -1) m_+(x, \tau )   \int _{y<x}    e^{{\bf i}\tau (x-y)} f(y)dy  \\& +  R_+ (-\tau ) m_+ (x,-\tau )  \int _{y>x}   e^{-{\bf i}\tau (x+y)}  f(y)dy\big ] .
\end{aligned}
\end{equation*}
We rewrite the above as

\begin{equation} \label{eq:b21} \begin{aligned} &
    B_j^{(2)}(f)  =- \frac  1 {2 \pi} 2^{- 2j} \int _{{\mathbf R}}  dx  V(x)\overline{f(x)}  \\& \times\int _{\mathbf R}    d\tau \psi(2^{- j}  \tau  ) \big \{ \big [ T( \tau ) -1 -R_+ (-\tau ) \big ] m_+(x, \tau )   \int _{y<x}    e^{{\bf i}\tau (x-y)} f(y)dy  \\& -  R_+ (-\tau )  \left (e^{-{\bf i}\tau x} m_+ (x,-\tau ) -e^{ {\bf i}\tau x} m_+ (x, \tau )\right ) \int _{y<x}  e^{-{\bf i}\tau y}  f(y)dy
    \\& +  R_+ (-\tau ) m_+ (x,-\tau )   e^{-{\bf i}\tau  x  }\int _{{\mathbf R}}   e^{-{\bf i}\tau  y }  f(y)dy\big \} .
\end{aligned}
\end{equation}
The last factor is $\sqrt{2\pi} \widehat{f}(-\tau )=0$ on the support of  $\psi(2^{- j}  \tau  )$ like after \eqref{eq:cancel}. So the last line in
\eqref{eq:b21}  cancels out.

\noindent We focus now on the terms originating from the third line of \eqref{eq:b21}.
We will set  $f_x(t):=f(t+x)$  and  $  {H{f}_x}(\tau ):= \int _{-\infty} ^0 e^{-{\bf i}\tau y} f(y+x) dy.$
We have
$$ {Hg}(\tau ) =  \int _{-\infty} ^0 e^{-{\bf i}\tau y} g(y) dy = \int_{{\mathbf R}} \widehat{\chi}_{(-\infty , 0]}(-\tau-\xi) \hat{g}(\xi) d\xi= \widehat{\chi }_{(-\infty , 0]} \ast  \widehat{g } (-\tau ),$$
where here and below we use the definition \eqref{eq:FT} of the Fourier transform.

We have also the relation $\widehat{\chi }_{(-\infty , 0]} (\tau )=-{\bf i}(2\pi )^{-\frac 12}  (\tau -{\bf i}0)^{-1}$, see page 206, Ch. 3 \cite{taylor} and take into account the definition of the Fourier transform there.
By Sokhotskyi-Plemelj formula $(\tau -{\bf i}0)^{-1}=P.V \frac 1 \tau +{\bf i}\pi \delta  (\tau )$. Then

\begin{equation} \label{eq:htran}  \begin{aligned}
 &  {Hg}(\tau ) =  \widehat{\chi }_{(-\infty , 0]} \ast  \widehat{g } (\tau ) = (2\pi)^{-1/2} \left( \pi  \widehat{g } (-\tau )  -i
\mathcal{H} g(-\tau )\right) \\&   \mathcal{H} g(\tau ) := \lim _{\epsilon \to 0^+} \int _{|\xi -\tau |\ge \epsilon}  \frac{\widehat{g } (\xi )}{\xi  -\tau }  d\xi .\end{aligned}
\end{equation}
By Lemma \ref{lem:Jost} we get
\begin{equation} \label{eq:bm}  \begin{aligned}
 &
 |e^{-{\bf i}\tau x} m_+ (x,-\tau ) -e^{ {\bf i}\tau x} m_+ (x, \tau )|\le |e^{-2{\bf i}\tau x}   -1|  \ |m_+ (x,\tau  )|\\&
+ |m_+ (x,-\tau ) -m_+ (x,\tau  )|\le
(C _1+C_2) \langle x\rangle ^2 |\tau |, \end{aligned}
\end{equation}
where the last term in the first line can be  bounded using \eqref{eq:kernel2} and the first term in the second line can be  bounded using the  mean value
theorem and  \eqref{eq:tderm}, and where $C_j=C(\|   V \| _{L^{1,j}})$.

\noindent By \eqref{eq:bm}  and by
 $|R_+ (-\tau )|\le  C \langle \tau \rangle  ^{-1}$   with $C=C(\|   V \| _{L^{1,1}}) $, which follows   by \eqref{eq:TRcoeff0},
 the   terms originating from the third line of \eqref{eq:b21} can be bounded   by a constant $C =C(\|   V \| _{L^{1,2}}) $  times
\begin{equation}\label{eq:b113}  \begin{aligned} &
      2^{- 2j} \| f\| _{L^\infty_x} \int _{{\mathbf R}}  dx \ | V(x)|\  \langle x \rangle   ^2 \int_{\mathbf R}  d\tau \ |\psi(2^{- j}  \tau  )|
\  |\tau |\ \langle \tau \rangle ^{-1}   |H  \widehat{{f}}_x(\tau )|.
\end{aligned}
\end{equation}
By $ |j-k| > \kappa _0$,
by $  \widehat{{f}}_x (\tau )=  e^{-{\bf i}\tau x}\widehat{{f}} (\tau )$ and by  \eqref{eq:htran},
we get  $\psi(2^{- j}  \tau  )  |H  \widehat{{f}}_x(\tau )|= \psi(2^{- j}  \tau  ) |\mathcal{H}  \widehat{{f}}_x(\tau )|$.
We have then the upper bound

\begin{equation*} \begin{aligned} &  |\text{\eqref{eq:b113} }|\le
2^{- 2j} \| f\| _{L^\infty_x} \|
V \| _{L^{1,2}}  \int _{|\tau |\sim 2^{j}}   d\tau \frac{|\tau |}{ \langle \tau  \rangle ^2}
\int _{|\xi |\sim 2^{k}}  \frac{ |\widehat{{f}}(\xi  )|}{|\tau -\xi |} d\xi \\&
\le 2^{-  j} \| f\| _{L^\infty_x} \|    V \| _{L^{1,2}}
\int _{|\xi |\sim 2^{k}}  { |\widehat{{f}}(\xi  )|}  d\xi
\le C'2^{\frac{k}{2}- j} \| f\| _{L^\infty_x}\| f\| _{L^2_x}
\\& \le C2^{k- j} \| f\| _{L^2_x}^2
\end{aligned}\nonumber
\end{equation*}
where we used $ |\tau -\xi |\approx |\tau   |$ and where $C=C(\|    V \| _{L^{1,2}})$.
Now we consider the contribution from the second line of \eqref{eq:b21}.
We assume
\begin{equation}\label{eq:essential}
T( 0 ) -1 -R_+ (0 )=0.
\end{equation}
\eqref{eq:essential} occurs if $T(0)=1$ (then $R_{\pm }(0)=0 $  by the identity \eqref{eq:TRcoeff})
and in   the generic case $T(0)=0$  (when $R_{\pm }(0)=-1 $, see p. 147 \cite{DT}, as can be seen setting $\tau =0$ in   \eqref{eq:kernel35}).
By  \eqref{eq:essential} and  \eqref{eq:TRcoeff1}
 for the bound near $\tau =0$ and by   \eqref{eq:TRcoeff0} for the bound away from 0,   we get
\begin{equation*}
|T( \tau  ) -1 -R_+ (-\tau  )|\le C \frac{|\tau |}{\langle \tau \rangle ^2} \text{ with $C=C(\| V \| _{L^{1,3}})$}.
\end{equation*}
Then, by a similar argument to that for the  third line   \eqref{eq:b21}
we see that the contribution is bounded by $C2^{k- j} \| f\| _{L^2_x}^2$ with $C=C(\| V \| _{L^{1,3}})$.
\end{proof}

\noindent Lemmas  \ref{lem:split1} and  \ref{lem:split2}  yield together Lemma \ref{l:diadhigh2}.
\end{proof}

\noindent The proof of Lemma \ref{l:diadhigh} follows by combining \eqref{eq:aj1} with Lemma \ref{l:diadhigh2}.

\end{proof}

We remark that if $T(0)=\frac{2a}{1+a^2}$  with $a\neq 0$ then $R_+(0)=\frac{1-a^2}{1+a^2}$, see for instance p. 512 \cite{We1}.
Then the rhs of \eqref{eq:essential} equals $2\frac{a-1}{1+a^2}\neq 0$ for $a\neq 1$
and our proof of Lemma \ref{lem:split2} breaks down.

We have proved \eqref{eq:quasidiag} for $k\le j$.  The next lemma shows that  \eqref{eq:quasidiag}
continues to hold also for  $k> j$
\begin{lemma}
\label{l:diadlow} Let $V$ be a real valued Schwartz function with  $
\sigma (\triangle _V )=(-\infty ,0]$ and with    $T(0)$  either equal to 0
or to 1.
For any  integer numbers $j,k \in {\mathbf Z} $ with $k > j$ and
for any $f \in {\mathcal S}( \mathbb{{R}}) $ satisfying \eqref{eq:suppf},
inequality \eqref{eq:quasidiag}  holds for a   $C_V$ of same type.

\end{lemma}

\begin{proof} The proof is similar to that of Lemma \ref{l:diadhigh}.

\noindent We have $f=\widetilde{\varphi} \left( 2^{-k} \sqrt{-\triangle  }\right) f $
for some $\widetilde{\varphi}\in C^\infty _0({\mathbf R} _+, [0,1])$
and   we have
\begin{equation*}   \begin{aligned} &  \langle \varphi \left(
2^{-j}\sqrt{-\triangle _V} \right)     f ,   f \rangle_{L^2_x}   =
-2^{-2k}\langle \varphi \left( 2^{-j}\sqrt{-\triangle _V} \right)   f ,
\triangle  \psi  \left(2^{-k}
\sqrt{-\triangle  } \right)f \rangle_{L^2_x}
\end{aligned}
\end{equation*}
with $\tau ^2 \psi (\tau ) =\widetilde{\varphi} (\tau )$. Then we have

\begin{equation*}   \begin{aligned}    \langle \varphi \left( 2^{-j}\sqrt{-\triangle _V} \right)     f ,   f \rangle_{L^2_x}
&=   -2^{-2k}\langle \triangle _V\varphi
\left(2^{-j}\sqrt{-\triangle _V} \right)   f ,
\psi  \left( 2^{-k} \sqrt{-\triangle  } \right)
f \rangle_{L^2_x} \\& -2^{-2k}\langle
V\varphi \left( 2^{-j}\sqrt{-\triangle _V} \right)   f ,
\psi  \left( 2^{-k} \sqrt{-\triangle  } \right)
f \rangle_{L^2_x}
\end{aligned}
\end{equation*}
It is straightforward that, for a constant $C$ independent of $V$,

 \begin{equation}\label{eq:high0}   \begin{aligned}
 2^{-2k}| \langle \triangle _V\varphi \left( 2^{-j}\sqrt{-\triangle _V} \right)   f ,
\psi  \left( 2^{-k}\sqrt{-\triangle  } \right)f \rangle_{L^2_x} |
\le C 2^{2j-2k} \| f\| ^2_{L^2_x}.
\end{aligned}
\end{equation}

\noindent In the sequel we prove the following for $C=C(\|V \|_{L^{1,3}})$,
which with \eqref{eq:high0} yields Lemma
 \ref{l:diadlow}:
    \begin{equation}\label{eq:high3}  \begin{aligned}& 2^{-2k}|\langle V\varphi
\left ( 2^{-j}\sqrt{-\triangle _V} \right ) f,  \psi  \left ( 2^{-k}\sqrt{-\triangle  } \right ) f \rangle_{L^2_x}|
\le C 2^{j-k} \|f\|^2_{L^2_x} . \end{aligned}\end{equation}
  Denote by $K (x,y)$ the integral kernel  of $\varphi\left ( 2^{-j}\sqrt{-\triangle _V} \right )$. Then,
	setting $g(\tau )= \varphi  \left ( 2^{-j}\sqrt{\tau } \right ) $,     from \eqref{eq:absprin4} we get

    \begin{equation}  \label{eq:high31} \begin{aligned}
 &    K (x,y) \sim   \chi _{ x>y} (x,y)    \int _{\mathbf R}   \varphi  (2^{- j}  \tau  )
   m_+(x, \tau )\, m_-(y, \tau )\,  T( \tau )\,  e^{{\bf i}\tau (x-y)} \\& + \chi _{ x<y} (x,y) \int _{\mathbf R}   \varphi  (2^{- j}  \tau  )  m_-(x, -\tau )m_+(y, -\tau )  T(- \tau )   \,  e^{{\bf i}\tau (x-y)} d\tau
     \end{aligned} \nonumber\end{equation}
     with $ \chi _{ x\gtrless y} (x,y)=1$ for $x\gtrless y$ and
     $ \chi _{ x\gtrless y} (x,y)=0$ for $x\lessgtr y$.
     Then the bound \eqref{eq:high0} is obtained,    for $\Psi (x)= \psi  \left( \frac{\sqrt{-\triangle  }}{2^k} \right) {f}$,  by bounding

\begin{equation}\label{eq:I2}  \begin{aligned} &  2^{-2k}  \int _{{\mathbf R}}  dx   \overline{\Psi (x)}  V(x)  \int _{\mathbf R}    d\tau \varphi(2^{- j}  \tau  )  \\& \times \big [ T( \tau )   m_+(x, \tau )   \int _{y<x}m_-(y, \tau )   e^{{\bf i}\tau (x-y)}  f (y)dy  \\& +   T(- \tau )  m_-(x, -\tau )  \int _{y>x} m_+(y, -\tau )  e^{{\bf i}\tau (x-y)} f (y)dy\big ] .
\end{aligned}
\end{equation}

\noindent We split \eqref{eq:I2}   as  $ I_{1}+I_{2} $ where
\begin{equation}\label{eq:I21}  \begin{aligned} &   I_1
 :=  2^{-2k} \int _{{\mathbf R}}  dx \overline{\Psi (x)}  V(x)    \int _{\mathbf R}    d\tau \varphi(2^{- j}  \tau  )  \\& \times \big [ T( \tau )   m_+(x, \tau )   \int _{y<x}\left ( m_-(y, \tau ) -1\right )  e^{{\bf i}\tau (x-y)}f (y)dy  \\& +   T(- \tau )  m_-(x, -\tau )  \int _{y>x}\left ( m_+(y, -\tau ) -1\right ) e^{{\bf i}\tau (x-y)} f (y)dy\big ]   \,
\end{aligned}
\end{equation} 
and

\begin{equation}\label{eq:I22}  \begin{aligned} &   I_2
    :=  2^{-2k} \int _{{\mathbf R}}  dx \overline{\Psi (x)}  V(x)    \int _{\mathbf R}    d\tau \varphi (2^{- j}  \tau  )  \\& \times \big [ T( \tau )  m_+(x, \tau )   \int _{y<x}   e^{{\bf i}\tau (x-y)} f (y)dy  \\& +   T(- \tau ) m_-(x, -\tau )  \int _{y>x}   e^{{\bf i}\tau (x-y)}f (y)dy\big ] .
\end{aligned}
\end{equation}
We  start with $I_1$   and show for $C=C(\| V \| _{L^{1,3}})$
    \begin{equation}\label{eq:I111}  \begin{aligned}
   &  |I_1| \le  C 2^{ j-   k}\| f\| _{L^2_x}^2
\end{aligned}
\end{equation}
 To prove \eqref{eq:I111} we  focus for definiteness on the
 second line of   \eqref{eq:I21} (the contribution from the third can be treated similarly). Then we have
      \begin{equation*}   \begin{aligned} & 2^{-2k} \int _{{\mathbf R}}  dx |{\Psi (x)}  V(x)|
\int _{\mathbf R}    d\tau  |\varphi(2^{- j}  \tau  )|\text{second line \eqref{eq:I21}}|\lesssim \\&
      2^{-2k} \int _{{\mathbf R} }  dx   | \Psi (x) V(x)|  \langle x\rangle   \int _{|\tau |\sim 2^j }    d\tau    \left ( \int _{-\infty}^{0\wedge x} \langle y\rangle ^{-2}  |f(y)|dy  +  \int _0 ^{x\vee 0}   \langle y\rangle    |f (y)|dy\right ) \\& \le C' 2^{ j-2k}
        \|  \Psi \| _{L^\infty_x}    \|   {f} \| _{L^\infty_x}\le C^{\prime\prime}  2^{ j-k }
\| \psi  \left( 2^{-k} \sqrt{-\triangle  _V} \right) {f} \| _{L^2_x}\|  {f} \| _{L^2_x} \\
& \le C   2^{ j-k } \|  {f} \| _{L^2_x} ^2
\end{aligned}
\end{equation*}
with constants  $ C(\| V \| _{L^{1,3}})$ and where we used Bernestein inequality \eqref{eq:bern}.
We turn now to $I_2$ and show  for $ C=C(\| V \| _{L^{1,3}})$\begin{equation}\label{eq:I222}  \begin{aligned}
   &  |  I_{2 } | \le  C 2^{j-  k}\| f\| _{L^2_x}^2.
\end{aligned}
\end{equation}
We substitute   \eqref{ker35bis}
to get

\begin{equation*}  \begin{aligned} & I_2  = 2^{-2k}  \int _{{\mathbf R}}  dx  \overline{\Psi (x)}V(x)  \\& \times\int _{\mathbf R}    d\tau \varphi (2^{- j}  \tau  ) \big [ (T( \tau ) -1)   m_+(x, \tau )   \int _{y<x}    e^{{\bf i}\tau (x-y)} f (y)dy  \\& +  R_+ (-\tau ) m_+ (x,-\tau )  \int _{y>x}   e^{-{\bf i}\tau (x+y)} f (y)dy \big ] .
\end{aligned}
\end{equation*}
We rewrite, proceeding like for  \eqref{eq:b21},

\begin{equation} \label{eq:I201} \begin{aligned} &
   I_2 = 2^{- 2k} \int _{{\mathbf R}}  dx \overline{\Psi (x)}V(x) \\& \times\int _{\mathbf R}    d\tau \varphi (2^{- j}  \tau  ) \big [ (T( \tau ) -1 -R_+ (-\tau ) ) m_+(x, \tau )   \int _{y<x}    e^{{\bf i}\tau (x-y)} f (y)dy  \\& -  R_+ (-\tau )  \left (e^{-{\bf i}\tau x} m_+ (x,-\tau ) -e^{ {\bf i}\tau x} m_+ (x, \tau )\right ) \int _{y<x}  e^{-{\bf i}\tau y}  f(y)dy
     \big ] .
\end{aligned}
\end{equation}
Then proceeding like in Lemma \ref{lem:split2} we get for $C=C(\| V \|  _{L^{1,3}})$

\begin{equation*}\label{eq:I213}  \begin{aligned} & |I_2|\le C
     2^{- 2k} \| \Psi \| _{L^\infty_x} \int _{{\mathbf R}}  dx | V(x)| \langle x \rangle  ^2 \int  d\tau |\varphi (2^{- j}  \tau  )|   \frac{|\tau |}{ \langle \tau  \rangle ^2}  |H  \widehat{{f}_x}(\tau )|
\end{aligned}
\end{equation*}
with  $  {H\widehat{f}_x}(\tau ):= \int _{-\infty} ^0 e^{-{\bf i}\tau y} f(y+x) dy $ like earlier. Since now we focus only on
   $k-j>\kappa _0$ and we get

\begin{equation*}\label{eq:I214}  \begin{aligned} &  | I_2|\le C_1
2^{- 2k} \| \Psi \| _{L^\infty_x} \| V \|  _{L^{1,2}}  \int _{|\tau |\sim 2^{j}}
d\tau \frac{|\tau |}{ \langle \tau  \rangle ^2}  \int _{|\xi |\sim 2^{k}}
\frac{ |\widehat{{f}}(\xi  )|}{|\tau -\xi |} d\xi \\&\le  C_22^{2 j-2k}  \| \Psi \| _{L^\infty_x}    2^{-k}  \int _{|\xi |\sim 2^{k}}  { |\widehat{{f}}(\xi  )|}  d\xi \\ & \le C_3 2^{2 j-2k} \| f\| _{L^2_x } 2^{-\frac{k}{2}}\| \psi  \left( {2^{-k}}  {\sqrt{-\triangle  }}
\right) {f}\| _{L^\infty_x }   \le C 2^{2 j-2k} \| f\| _{L^2_x }^2.
\end{aligned}
\end{equation*}
where the constants are  $ C(\| V \|  _{L^{1,3}})$.  This completes the proof of \eqref{eq:I222} which, along with \eqref{eq:I111} yields \eqref{eq:high3} and completes the proof of Lemma \ref{l:diadlow}.

\end{proof}

>From Lemma \ref{l:diadlow} and Lemma \ref{l:diadhigh} we arrive at the following crucial result.

\begin{corollary}
\label{l:ineperteq}
For $0 \leq s < 1/2$ and for any $f \in C_0^\infty ({\mathbf R})$   we have
$$ \|(-\triangle )^{\frac s2} f \|_{L^2_x} \sim  \|(-\triangle +V)^{\frac s2} f \|_{L^2_x} . $$

\end{corollary}
\begin{proof} The proof of $\gtrsim$ is as follows (that of $\lesssim$ is similar).
We have

\begin{equation} \begin{aligned} &  \|(-\triangle +V)^{\frac s2 } f \|_{L^2_x}^2
\\& \sim \sum_{j ,k,l\in {\mathbf Z}} 2^{2js} \langle   \varphi
\left(  \frac{\sqrt{-\triangle _V }}{2^j}\right)  \varphi \left( \frac{\sqrt{-\triangle  }}{2^k}
\right)  f ,
\varphi \left(\frac{\sqrt{-\triangle _V }}{2^j}\right)
\varphi \left( \frac{\sqrt{-\triangle  }}{2^l}\right)  f
\rangle_{L^2_x} \\& \le C   \sum_{j ,k,l\in {\mathbf Z}} 2^{2js}
2^{-\frac{1}{2}|j-k|-\frac{1}{2}|j-l|} \| \varphi \left(  \frac{\sqrt{-\triangle  }}{2^k}\right) f \| _{L^2_x}  \| \varphi \left(  \frac{\sqrt{-\triangle  }}{2^l}\right) f \| _{L^2_x}   \\& \le C' \sum_{ k\in {\mathbf Z}}  2^{2ks}
\| \varphi \left(  \frac{\sqrt{-\triangle  }}{2^k}\right) f \| _{L^2_x} ^2
\sim C'\|(-\triangle  )^{\frac s2 } f \|_{L^2_x}^2.
\end{aligned} \nonumber
\end{equation}
Here we have used Young's inequality and, for a fixed $C$,
 \begin{equation} \begin{aligned} &    2^{-ls} \sum_{j ,k } 2^{2js} 2^{-\frac{1}{2}|j-k|-\frac{1}{2}|j-l|} 2^{-ks}\\& = 2^{-ls} \sum_{j   } 2^{2js} 2^{ -\frac{1}{2}|j-l|}  \big [ 2^{\frac{j}{2}}\sum_{ k\ge j }    2^{-ks-\frac{k}{2}}  + 2^{-\frac{j}{2}}\sum  _{ k\le j-1 }  2^{ \frac{k}{2}  -ks}    \big ] \\& \sim  2^{-ls} \sum_{j   } 2^{2js} 2^{ -\frac{1}{2}|j-l|}  2^{-js}
 =  2^{\frac{l}{2}  -ls }\sum_{j\ge l   }  2^{ js -\frac{1}{2}j} + 2^{-\frac{l}{2}-ls}\sum_{j\le l-1   }  2^{ js +\frac{1}{2}j}  \le C.
\end{aligned} \nonumber
\end{equation}

\end{proof}

\begin{remark}
\label{rem:discrmodes} The proof of Corollary \ref{l:ineperteq} continues to hold also when  from Hypothesis (H) we drop the requirement that
$\sigma   (\triangle _V) =(-\infty , 0] $  but for $f$
  we require additionally
$\langle f , \phi \rangle _{L^2}=0$ for all eigenfunctions $\phi $ of $\triangle _V$.
\end{remark}

\section{Proof of Lemma \ref{lem:boundA}}
\label{sec:boundA}

\begin{lemma}
  \label{lem:kato}   For $V_1=2V+ x\frac{d}{dx}V$, for $A(s)$ the operator
  in \eqref{JVs} and for $0< s<2$ we have for a constant $c(s)$

  \begin{equation}\label{eq:kato1}\begin{aligned} & A(s)=c(s) \int _0^\infty
\tau^{\frac{s}{2}}  (\tau - \triangle _V) ^{-1}V_1 (\tau - \triangle _V) ^{-1}
d\tau .
\end{aligned}  \end{equation}

\end{lemma}
\begin{proof} Set $S:=x\partial _x$.
Recall the formula
\begin{equation*}\label{eq:kato2} \begin{aligned} &   (- \triangle_V)^{\frac{s}{2}} =c(s)(- \triangle _V) \int _0^\infty
\tau^{\frac{s}{2}-1}  (\tau - \triangle _V) ^{-1}
d\tau      \end{aligned}
 \end{equation*}
for $0<s<2$  and $[c(s)]^{-1}=\int _0^\infty
\tau^{\frac{s}{2}-1}  (\tau  +1) ^{-1}
d\tau    $.
Then \begin{equation}\label{eq:kato11}\begin{aligned} &
    A(s)= { s} (- \triangle_V)^{\frac{s}{2}}+ c(s) \int _0^\infty \tau^{\frac{s}{2}-1}  \left[S ,- \triangle _V  (\tau - \triangle _V) ^{-1} \right]  d\tau.   \end{aligned}
 \end{equation}
 We have

\begin{equation*} \label{eq:kato3}\begin{aligned} &  \left[S ,- \triangle _V  (\tau - \triangle _V) ^{-1} \right] =\left[S ,- \triangle _V   \right] (\tau - \triangle _V) ^{-1}  - \triangle _V  \left[S , (\tau - \triangle _V) ^{-1} \right]  =\\& \left[S ,- \triangle _V   \right] (\tau - \triangle _V) ^{-1}  + \triangle _V  (\tau - \triangle _V) ^{-1}\left[S ,  - \triangle _V  \right] (\tau - \triangle _V) ^{-1}    \end{aligned}
 \end{equation*}
and also
\begin{equation*} \label{eq:kato4}\begin{aligned} &  \left[S , - \triangle_V \right] =  \left[S , - \triangle  \right] + \left[S ,  V \right]  =2\triangle +SV = 2(\triangle   -V) +V_1= 2 \triangle _V  +V_1.   \end{aligned}
 \end{equation*}
Then we get
\begin{equation} \label{eq:kato5}\begin{aligned} &  \left[S ,- \triangle _V  (\tau - \triangle _V) ^{-1} \right] =   2 \triangle _V     (\tau - \triangle _V) ^{-1} +2 \triangle _V ^2 (\tau - \triangle _V) ^{-2} \\& +
  V _1    (\tau - \triangle _V) ^{-1}
 + \triangle _V  (\tau - \triangle _V) ^{-1} V_1 (\tau - \triangle _V) ^{-1} \\& =2 \tau \triangle _V  (\tau - \triangle _V) ^{-2} +\tau (\tau - \triangle _V) ^{-1}   V _1    (\tau - \triangle _V) ^{-1}  .   \end{aligned}
 \end{equation}
 Inserting in \eqref{eq:kato11} we get
 \begin{equation}\label{eq:kato6}\begin{aligned} &
    A(s)= { s} (- \triangle_V)^{\frac{s}{2}}+ 2c(s) \int _0^\infty \tau^{\frac{s}{2} } \triangle _V  (\tau - \triangle _V) ^{-2}  d\tau \\& +  c(s) \int _0^\infty \tau^{\frac{s}{2} }  (\tau - \triangle _V) ^{-1}   V _1    (\tau - \triangle _V) ^{-1}  d\tau.   \end{aligned}
 \end{equation}
 Then  \eqref{eq:kato1} follows from the fact that the first line of the r.h.s. is 0: for $y>0$ we have integrating by parts
\begin{equation*}\label{eq:kato7}\begin{aligned} &
    -2c(s) y\int _0^\infty \tau^{\frac{s}{2} }  (\tau +y) ^{-2}  d\tau   =- 2c(s)y \frac{s}{2} \int _0^\infty \tau^{\frac{s}{2}-1 }  (\tau +y)^{-1}   d\tau = -s y^{\frac{s}{2} }.\end{aligned}
 \end{equation*}

\end{proof}

\begin{lemma}
  \label{lem:resolvent} Given Hypothesis (H)
   there is a fixed $C= C(\| V \|  _{L^{1,1}})$ such that for any $f\in {\mathcal S}({\mathbf R} )$
  and at any $x\in {\mathbf R}$ we have

\begin{equation} \label{eq:resolvent1}\begin{aligned} &
\left  | \left [(\tau - \triangle _V) ^{-1}  f\right ] (x) \right |
\le C \langle \tau \rangle ^{-\frac{1}{2}} \int _{{\mathbf R}}e^{-\sqrt{\tau}|x-y|} \langle y \rangle
\left  |   f (y) \right | dy.
\end{aligned}\end{equation}
\end{lemma}
\begin{proof} Consider the Wronskian $w(\sqrt{\tau}) $ defined in \eqref{eq:wronskian} . Recall that, since $V\in {\mathcal S}({\mathbf R})$, we have
$w(\sqrt{\tau})>0$ for $ {\tau}>0$ and
$w(\sqrt{\tau})\sim \sqrt{\tau}$ as  $ {\tau}\to +\infty$. The hypothesis
that $T(0)=0$ implies that  $w(0)>0$.

\noindent We have \begin{equation*}  \begin{aligned} &
 \left [(\tau - \triangle _V) ^{-1}  f\right ] (x)
 =\int _{-\infty}^x \frac{m_{
+ } (x,\sqrt{\tau} )m_{
- } (y,\sqrt{\tau} )}{w(\sqrt{\tau})}e^{-\sqrt{\tau}|x-y|}
     f (y)   dy \\& +\int ^{+\infty}_x \frac{m_{
+ } (y,\sqrt{\tau} )m_{
- } (x,\sqrt{\tau} )}{w(\sqrt{\tau})}e^{-\sqrt{\tau}|x-y|}
    f (y)   dy .
\end{aligned}\end{equation*}
We will use $0<w^{-1}(\sqrt{\tau})<C_1\langle \tau \rangle ^{-\frac{1}{2}} $
for a fixed  $C_1= C(\| V \|  _{L^{1,1}})$. Inequality \eqref{eq:resolvent1} follows in elementary
fashion by   the following inequalities, where $C_2= C(\| V \|  _{L^{1,1}})$  is
a fixed sufficiently large number:

\begin{itemize}
\item  for $x\ge 0$ we have $|m_{
+ } (x,\sqrt{\tau} )m_{
- } (y,\sqrt{\tau} )|\le C_2\langle y \rangle $;

\item  for $x\ge 0$ we have $|m_{
- } (x,\sqrt{\tau} )m_{
+ } (y,\sqrt{\tau} )| \chi _{{\mathbf R}^+} (y-x)\le C_2\langle x \rangle \le C_2\langle y \rangle $;

\item  for $x< 0$ we have $|m_{
+ } (x,\sqrt{\tau} )m_{
- } (y,\sqrt{\tau} )| \chi _{{\mathbf R}^+} (x-y)\le C_2\langle x \rangle \le C_2\langle y \rangle $;

\item  for $x< 0$ we have $|m_{
- } (x,\sqrt{\tau} )m_{
+ } (y,\sqrt{\tau} )|   \le C_2\langle y \rangle $.

\end{itemize}

\end{proof}

\begin{lemma}
\label{lem:resolvent1} Under Hypothesis (H) there is a fixed $C $ such that

\begin{equation} \label{eq:resolvent11}\begin{aligned} &
\|  (\tau - \triangle _V) ^{-1}   V _1    (\tau - \triangle _V) ^{-1} f \| _{L^1_x}
\\&\le C   \left ( \| V \|  _{L^{1,2} } + \| \frac  d{dx}V \|  _{L^{1,3} } \right ) \tau   ^{-1} \langle \tau \rangle ^{-1} \| f \| _{L^\infty_x}.
\end{aligned}\end{equation}
\end{lemma}
\begin{proof}
We can factorize $V_1= \langle x \rangle ^{-2}V_2$ with  
$  V_2 \in L^1 ({\mathbf R} ) $.
\noindent We have \begin{equation*}  \begin{aligned} &
 \|  (\tau - \triangle _V) ^{-1}   V _1    (\tau - \triangle _V) ^{-1} f \| _{L^1_x} \\&
\le \|  (\tau - \triangle _V) ^{-1}   \langle x \rangle ^{-1}  \| _{L^1_x\to L^1_x}    \|  V _2  \| _{L^1_x } \|   \langle x \rangle ^{-1} (\tau - \triangle _V) ^{-1}    \| _{L^\infty_x \to L^\infty_x}  .
\end{aligned}\end{equation*}
We have
\begin{equation*}  \begin{aligned} &
   \|  (\tau - \triangle _V) ^{-1}    \langle  \cdot  \rangle ^{-1} f \| _{L^1_x }   \le
C    \langle \tau \rangle ^{-\frac{1}{2}}  \| e^{-\sqrt{\tau}|\cdot |}* ( \langle \cdot  \rangle
  \left  | \langle  \cdot  \rangle ^{-1} f (\cdot ) \right | ) \| _{L^1_x }\\&  \le C'
\tau   ^{-\frac{1}{2}}\langle \tau \rangle ^{-\frac{1}{2}}    \|f\| _{L^1_x} .
\end{aligned}\end{equation*}
 The following bound with the same $C'$ follows by duality:
 \begin{equation*}  \begin{aligned} &
   \|   \langle x \rangle ^{-1}  (\tau - \triangle _V) ^{-1} f \| _{L^\infty_x }   \le   C'  \tau   ^{-\frac{1}{2}}\langle \tau \rangle ^{-\frac{1}{2}}     \|f\| _{L^\infty_x } .
\end{aligned}\end{equation*} 
Finally, by $V_2=\langle x \rangle ^{ 2}  (2V +x V' ) $ it follows that $\|  V _2  \| _{L^1  }\lesssim   \| V \|  _{L^{1,2} } + \| \frac  d{dx}V \|  _{L^{1,3} } .$

This yields inequality \eqref{eq:resolvent11}.
\end{proof}

\bigskip \textit{Proof of Lemma \ref{lem:boundA}}. The inequality  $\| A(s)f\| _{L^1_x} \le C  \| f\| _{L^\infty_x}    $ for fixed $C>0$ follows
by Lemmas \ref{lem:kato} and \ref{lem:resolvent1} which justify the following inequalities:

\begin{equation*}  \begin{aligned} &
   \| A(s)f\| _{L^1_x} \le c(s) \int _0^\infty
\tau^{\frac{s}{2}}  \| (\tau - \triangle _V) ^{-1}V_1 (\tau - \triangle _V) ^{-1}f\| _{L^1_x}
d\tau  \\& \le C ' \| f\| _{L^\infty_x} \int _0^\infty
\tau^{\frac{s}{2}-1} \langle \tau \rangle ^{-1}  d\tau \le C    \| f\| _{L^\infty_x}
\end{aligned}\end{equation*}
where the integral converges if $0<s<2$   and where $C = C(s,\| V \|  _{L^{1,2}}, \| V '\|  _{L^{1,3}} )$. \qed

Department of Mathematics and Geosciences,  University of Trieste, via Valerio  12/1  Trieste, 34127  Italy

{\it E-mail Address}: {\tt scuccagna@units.it}

\vskip10pt

Department of Mathematics,  University of Pisa,    Largo Bruno Pontecorvo 5  Pisa, 56127  Italy.

{\it E-mail Address}: {\tt georgiev@dm.unipi.it}

{\it E-mail Address}: {\tt viscigli@dm.unipi.it}


\begin{thebibliography}{xxxxxx 89}




\bibitem{B}  Barab, J.E.  Nonexistence of asymptotic free solutions for a nonlinear
equation. {\em J. Math. Phys.}  {\bf 25} (1984), 3270 -- 3273.


\bibitem{DT} Deift, P.;  Trubowitz, E.    Inverse scattering on
the line.  {\em Comm. Pure Appl. Math.}  {\bf 32}, (1979) 121--251.


\bibitem{DZ} Deift, P.;    Zhou, X.   Long-time asymptotics for solutions of the NLS equation with initial data in a weighted Sobolev space. { \em Comm. Pure Appl. Math.}  {\bf 56} (2003), 1029--1077. 

\bibitem{DZ1} Deift, P.;    Zhou, X.  Perturbation theory for infinite-dimensional integrable systems on the line. A case study.  { \em Acta Math.}  {\bf 188} (2002), 163–262.

\bibitem{DZ3} 
Deift, P.; Zhou, X. A steepest descent method for oscillatory Riemann-Hilbert problems. Asymptotics for the MKdV equation. Ann. of Math.   {\bf 137} 137 (1993), no. 2, 295–368


\bibitem{DZ2} Deift, P.;    Zhou, X.  Long-time behavior of the non-focusing nonlinear Schr\"odinger equation—a case study. New Series: Lectures in Mathematical Sciences, 5. University of Tokyo, 1994.

 
\bibitem{DP} Deift, P.;  Park, J.   Long-time asymptotics for solutions of the NLS equation with a delta potential and even initial data. { \em Int. Math. Res. Not. IMRN} {\bf 24} (2011) , 5505--5624.


\bibitem{DM}Dieng,  M.; McLaughlin,K. D. T.-R.  
Long-time Asymptotics for the NLS equation via $\overline{\partial} $ methods,   	arXiv:0805.2807.


\bibitem{GO}  Ginibre, J.; Ozawa, T.
Long range scattering for nonlinear Schr\"odinger and Hartree equations in space dimension $n\ge 2$.  { \em Comm. Math. Phys.}  {\bf 151} (1993), 619 -- 645.


\bibitem{GY} 
Galtabiar, A.;  Yajima,  K.  $L^p$ boundedness of wave operators
for one dimensional Schr\"odinger
operators. { \em J. Math. Sci. Univ. Tokio}  {\bf 7} (2000), 221 -- 240.
 

\bibitem{G73}  Glassey, R.   On the assymptotic behavior of nonlinear wave equations. { \em 
Trans. AMS.}  {\bf 182} (1973), 187--200.


\bibitem{HN} Hayashi, N.;  Naumkin, P.  Asymptotics for large time of solutions to the nonlinear Schr\"odinger and Hartree equations,  { \em Amer. J. Math.}  {\bf  120 } (1998) , 369--389.

\bibitem{Klainerman} Klainerman, S. Remarks on the Global Sobolev Inequalities in the Minkowski Space ${\mathbf R}^{n+1}$.  { \em  Comm. Pure Appl. Math.}  {\bf 40} (1987) , 111--117.


\bibitem{MS} McKean, H.;  Shatah, J. The nonlinear Schr\"odinger equation and the nonlinear heat equation reduction to linear form.  { \em  Comm. Pure Appl. Math. }  {\bf 44}   (1991) , 1067--1080.



 \bibitem{O91}  Ozawa, T.  Long range scattering for nonlinear Schr\"odinger equations in one space dimension. {\em Comm. Math. Phys.} {\bf  139 } (1991),  479 -- 493.


\bibitem{S74}  Strauss, W. Nonlinear scattering theory, {\em  Scattering Theory in Mathematical
Physics} (1974) 53--78.

\bibitem{S81} Strauss, W. Nonlinear scattering theory at low energy:  sequel.    {\em  Jour. Funct. Analysis} {\bf 43} (1981),  281--293.


\bibitem{taylor}
 Taylor, M.E. {\em Partial Differential Equations I \/}, Texts. Appl.
Math.  23, Springer, New York (1996).





\bibitem{We1} Weder, R.  The $W^{k,p}$ continuity of the Schr\"odinger
  wave operators on the line. { \em Comm. Math. Phys.}    {\bf 208} (1999), 507--520.


\bibitem{We2}   Weder, R.    $L^p\to L^{p^\prime}$ estimates for
 the Schr\"odinger equation
   on the line and inverse
scattering for the nonlinear
Schr\"odinger equation with
a potential. { \em
    J. Funct. Anal.}  {\bf  170} (2000),
  37--68.


\end{thebibliography}
\end{document}